\documentclass[12pt]{article}
\usepackage[T2A]{fontenc}
\usepackage[utf8]{inputenc}
\usepackage{amssymb,amsmath,amsfonts,amsthm,amscd,latexsym,verbatim,indentfirst,hyperref}
\hypersetup{hidelinks}

\usepackage{ulem}

\usepackage{stmaryrd}
\usepackage{xcolor}
\usepackage{geometry}
\def\id{\mathrm{id}}
\def\sl{\mathrm{sl}}

\def\Cur{\mathrm{Cur}}
\def\Span{\mathrm{Span}}
\def\Spec{\mathrm{Spec}}

\def\Aut{\mathrm{Aut}}

\newtheorem{lemma}{Lemma}
\newtheorem{theorem}{Theorem}
\newtheorem{proposition}{Proposition}
\newtheorem{corollary}{Corollary}
\newtheorem{example}{Example}

\begin{document}

\begin{center}
{\Large
Rota---Baxter operators on $\Cur(\sl_2(\mathbb{C}))$
}

\smallskip

Vsevolod Gubarev, Roman Kozlov
\end{center}

\begin{abstract}

We classify all Rota---Baxter operators on the simple conformal Lie algebra $\Cur(\sl_2(\mathbb{C}))$ and clarify which of them arise from the solutions to the conformal classical Yang---Baxter equation due to the connection discovered by Y. Hong and C. Bai in 2020.

\end{abstract}

\section{Introduction}

A linear operator~$R$ defined on an algebra~$A$ is called a Rota---Baxter operator, if the following relation holds for all $x,y\in A$,
$$
R(a)R(b) = R(R(a)b + aR(b) + kab).
$$
Here $k$ is a fixed scalar from the ground field, which is called a weight of~$R$. 
Nowadays, we observe the growing interest to such operators.
Defined by G. Baxter in 1960~\cite{Baxter} as an abstract generalization of the integral operator, Rota---Baxter operators showed their importance due to close connection with the Yang---Baxter equation~\cite{Aguiar,Belavin1982}, pre- and postalgebras~\cite{BBGN2011,Embedding} (among them pre-Lie algebras (also called left-symmetric algebras) are of the most interest), double Lie algebras~\cite{DoubleLie2,DoubleLie}.
For more details see the monograph of Li Guo~\cite{GuoMonograph}.

A notion of Lie conformal algebra introduced by V.G. Kac in~\cite{Kac} is an important tool to study vertex operator algebras. 
In turn, vertex algebras describe algebraic properties of the operator product expansion (OPE) in the two-dimensional conformal field theory developed by A.A.~Belavin, A.M. Polyakov and A.B. Zamolodchikov~\cite{BPZ} in 1983. In 1986, R. Borcherds confirmed the deep connection between vertex algebras and the Monster group~\cite{Borcherds}.
To the moment, vertex algebras form an actively studied area related to many others such as representation theory and mathematical physics, see~\cite{FB,LL}. 

In 2008, J. Liberati introduced~\cite{Liberati} a notion of Lie conformal bialgebra and suggested so called conformal classical Yang---Baxter equation as a source of (coboundary) Lie conformal bialgebras.
In 2012, this theory was extended by C.~Boyallian and J.~Liberati for Lie pseudobialgebras~\cite{BL}. 

In 2020, Y. Hong and C. Bai introduced~\cite{Bai} a~notion of Rota---Baxter operator on a~Lie conformal algebra and showed that every skew-symmetric solution to the conformal classical Yang---Baxter equation on a Lie conformal algebra~$L$ endowed with an invariant bilinear nondegenerate form gives rise to a~Rota---Baxter operator on~$L$.
This result reproduces the well-known connection between Rota---Baxter operators and solutions to the classical Yang---Baxter equation on a finite-dimensional semisimple Lie algebra~\cite{Belavin1982}.
A~connection between associative algebras and associative Yang---Baxter equation, which does not involve any kind of form, was found by M. Aguiar~\cite{Aguiar}.

Pseudoalgebras defined by B. Bakalov, A. D’Andrea and V.G.~Kac in~\cite{BAK} serve as a~natural generalization of conformal algebras involving a cocommutative Hopf algebra~$H$. For $H = \{e\}$, we get ordinary algebras and for $H = F[\partial]$ we obtain exactly conformal algebras.
A notion of Rota---Baxter operator on ($H$-)pseudoalgebras was suggested by L. Liu and S. Wang in 2020~\cite{Liu2020}. 
Note that this definition applied to conformal algebras differs from the one given by Y.~Hong and C.~Bai.

Properties and cohomologies of Rota---Baxter operators on Lie conformal algebras were recently studied in~\cite{Liu2022',Zhao2021}.
In~\cite{Yuan2022}, cohomologies of associative conformal Rota---Baxter algebras were defined.

In~\cite{SimpleConf}, A. D'Andrea and V.G. Kac proved the following structure result: every simple Lie conformal algebra of finite type
is isomorphic either to the Virasoro conformal algebra $\mathrm{Vir}$ or to the current Lie conformal algebra $\Cur(\mathfrak{g})$ associated to a simple finite-dimensional Lie algebra $\mathfrak{g}$.
In~\cite{Bai}, the authors showed that there are no nontrivial Rota---Baxter operators on $\mathrm{Vir}$.

The main goal of the current work is to describe all Rota---Baxter operators on  the remaining simple Lie conformal Lie algebra $\Cur(\sl_2(\mathbb{C}))$ of rank~1. We solve this problem completely, dealing with zero and nonzero cases separately. 
Since the classification of Rota---Baxter operators on $\sl_2(\mathbb{C})$ is known~\cite{Kolesnikov,KonovDissert,sl2}, 
we can trivially extend them onto $\Cur(\sl_2(\mathbb{C}))$. However, we find other Rota---Baxter operators, which depend on choice of a polynomial~$q(\partial)$.

In~\cite{Cursl2}, all solutions to the classical Yang---Baxter equation on $\Cur(\sl_2(\mathbb{C}))$ were described.
Hence, we may apply the connection established by Y. Hong and C. Bai to get Rota---Baxter operators of weight~0 on $\Cur(\sl_2(\mathbb{C}))$ corresponding to these solutions. 
Doing this, we get all Rota---Baxter operators of weight zero that are defined with an odd polynomial $q(\partial)$.

Let us give a short outline of the work.
In~\S2, we list the required preliminaries on Lie conformal algebras, the conformal classical Yang---Baxter equation and Rota---Baxter operators.
For our purposes, we prove more explicit classification of Rota---Baxter operators of nonzero weight on~$\sl_2(\mathbb{C})$.
In~\S3, we firstly describe Rota---Baxter operators of weight~0 on $\Cur(\sl_2(\mathbb{C}))$ and then show which of them come from the solutions to the conformal classical Yang---Baxter equation.
In~\S4, we classify Rota---Baxter operators of nonzero weight on $\Cur(\sl_2(\mathbb{C}))$.
We also provide examples of Rota---Baxter operators defined on any 
finite-dimensional semisimple Lie algebra over $\mathbb{C}$.

Throughout the work, all algebras and vector spaces are considered over the field of complex numbers.

\section{Preliminaries}

\subsection{Rota---Baxter operators on Lie algebras}

Let $\mathfrak{g}$ be a Lie algebra. A linear operator $R$ on $\mathfrak{g}$ is called a~Rota---Baxter operator
(RB-operator, for short) of weight~$k\in\mathbb{C}$ if

\vspace{-0.3cm}
\begin{equation}\label{RBusual}
[R(a),R(b)]
 = R([R(a),b] + [a,R(b)] + k[a,b])
\end{equation}
\vspace{-0.6cm}

\noindent
for all $a,b\in \mathfrak{g}$.

Note that given an RB-operator $R$ of weight~0 on a Lie algebra $\mathfrak{g}$,
the linear operator $\alpha R$ is again an RB-operator of weight~0 on $\mathfrak{g}$
for any $\alpha\in\mathbb{C}$.

Given a Lie algebra~$L$, we call the operators $R = 0$ and $R = -k\cdot\id$
as trivial Rota---Baxter operators of weight~$k$ on $L$.

\begin{lemma}[\cite{GuoMonograph}]\label{lem:phi}
Let $R$ be an RB-operator of weight~$k$ on a Lie algebra~$\mathfrak{g}$, then 
$\phi(R) = -(R+\lambda\id)$ is again an RB-operator of weight~$k$ on~$\mathfrak{g}$.
\end{lemma}

\begin{lemma}[\cite{BGP}]
Let $\mathfrak{g}$ be a Lie algebra, $\varphi$ be an automorphism of $\mathfrak{g}$
and $R$ be an RB-operator of weight~$k$ on $\mathfrak{g}$.
Then $R^{(\varphi)} = \varphi^{-1}R\varphi$
is an RB-operator of weight~$k$ on $\mathfrak{g}$.
\end{lemma}

Let a Lie algebra $\mathfrak{g}$ splits as a vector space
into direct sum of two subalgebras $\mathfrak{g}_1$ and~$\mathfrak{g}_2$.
An operator $R$ defined as
\begin{equation}\label{Split}
R(a_1 + a_2) = -k a_2,\quad a_1\in \mathfrak{g}_1,\ a_2\in \mathfrak{g}_2,
\end{equation}
is an RB-operator of weight~$k$ on~$\mathfrak{g}$ called a splitting RB-operator~\cite{GuoMonograph}.

\begin{example}
Let $\mathfrak{g}$ be a finite-dimensional semisimple Lie algebra over $\mathbb{C}$
with a~root system~$\Phi$.
Let a linear operator $R$ acts on $\mathfrak{g}$ as follows,
$R(\mathfrak{h}) = 0$ for the~fixed Cartan subalgebra $\mathfrak{h}$ of $\mathfrak{g}$, $R(e_\lambda) = -e_\lambda$, when $\lambda\in\Phi_+$, and $R(e_\lambda) = 0$ for all $\lambda\in\Phi_-$.
Then $R$ is a splitting Rota---Baxter operator of weight~1 on $\mathfrak{g}$.
\end{example}

Throughout this paper let us fix the standard basis $e,f,h$ of $\sl_2(\mathbb{C})$ such that
$$
[e,f] =  h, \quad
[h,e] = 2e, \quad
[h,f] = -2f.
$$

\begin{proposition}[\cite{Kolesnikov}] \label{prop:RBOnsl2-0}
Up to conjugation with an automorphism of $\sl_2(\mathbb{C})$ and up to a~scalar multiple,
we have that a~Rota---Baxter operator $R$ of weight 0 on $\sl_2(\mathbb{C})$
is one of the following:

a) $R \equiv 0$,

b) $R(e) = 0$, $R(f) = te - h$, $R(h) = 2e$, $t\in\mathbb{C}$,

c) $R(e) = 0$, $R(f) = 0$, $R(h) = h$,

d) $R(e) = 0$, $R(f) = h$, $R(h) = 0$,

e) $R(e) = 0$, $R(f) = e$, $R(h) = 0$.
\end{proposition}

The following result was actually proven in~\cite{KonovDissert} and in~\cite{sl2}.
Since it was not stated in the form we are interested, let us prove it directly.

\begin{proposition}[\cite{KonovDissert,sl2}] \label{prop:RBOnsl2-1}
Up to conjugation with an automorphism, all nontrivial RB-operators of weight~1 on $\mathrm{sl}_2(\mathbb{C})$ are the following:

a) $R(e) = -e$, $R(f) = R(h) = 0$,

b) $R(e) = -(e + h)$, $R(f) = R(h) = 0$,

c) $R(e) = -e$, $R(f) = 0$, $R(h) = th$, $t\in\mathbb{C}^*$.
\end{proposition}

\begin{proof}
Let $R$ be an RB-operator of weight~1 on $\sl_2(\mathbb{C})$.
If either $\ker(R) = (0)$ or $\ker(R+\id) = (0)$, then
$R$~is trivial~\cite{BurGub}.

Suppose that $R$ is splitting, i.\,e. 
$\sl_2(\mathbb{C}) = A\oplus B$, where 
$A$ and $B$ are subalgebras of $\sl_2(\mathbb{C})$
and $A = \ker R$, $B = \ker(R+\id)$.
We may assume that $\dim A = 2$ and $\dim B = 1$.
It is known that up to action of $\Aut(\sl_2(\mathbb{C}))$,
$A = \Span\{f,h\}$. Thus, $B = \Span\{e + kf + lh\}$ for some $k,l\in\mathbb{C}$.
Let $k = 0$, then for $l = 0$ it is~a).
When $l\neq0$ we conjugate with 
$\psi\in\Aut(\sl_2(\mathbb{C}))$ defined as follows,
$\psi(e) = (1/l)e$, $\psi(f) = lf$, $\psi(h) = h$,
and get~b) for $R^{(\psi)}$.

Suppose that $k\neq0$.
Consider $\psi\in \Aut(\sl_2(\mathbb{C}))$ defined as follows,
$$
\psi(e) = e - (k-\alpha l)f - (\alpha/2)h,\quad
\psi(f) = f,\quad
\psi(h) = h + \alpha f,\quad
\alpha = 2(-l+\sqrt{l^2+k}).
$$
Then for the RB-operator $R' = \psi R\psi^{-1}$, we have $\ker R' = \Span\{f,h\}$
and $\ker(R'+\id) = \Span\{e + \sqrt{l^2+k}h\}$, it is the already considered case.

Suppose that $R$ is not splitting, then $\dim(\ker R) = \dim(\ker(R+\id))=1$. It is known (see, e.\,g.~\cite{BurGub}) that
the space $\sl_2(\mathbb{C})$ under the new product
\begin{equation} \label{DesProduct}
\langle x,y\rangle := [R(x),y] + [x,R(y)] + [x,y] 
\end{equation}
is again a Lie algebra, and $R,R+\id$ are homomorphisms from
$(\sl_2(\mathbb{C}),\langle,\rangle)$ to $(\sl_2(\mathbb{C}),[,])$.
Hence, $\ker R,\ker(R+\id)$ are ideals in $(\sl_2(\mathbb{C}),\langle,\rangle)$.

Since we work over $\mathbb{C}$, let us consider all possible cases of the Jordan form of $R$. Let $e_2,e_3\in \sl_2(\mathbb{C})$ be such that 
$R(e_2) = 0$ and $R(e_3) = -e_3$.

{\sc Case I}: $\Spec(R) = \{0,-1\}$. 
Up to action of $\phi$, we may assume that there exists $e_1$ such that $R(e_1) = e_2$. Since both kernels are ideals in 
$(\sl_2(\mathbb{C}),\langle,\rangle)$,
we derive by~\eqref{DesProduct},
$$
\langle e_1,e_2\rangle = [e_1,e_2] = \mu e_2, \quad
\langle e_1,e_3\rangle = [e_2,e_3] = \lambda e_3
$$
for some (nonzero) $\lambda,\mu\in \mathbb{C}$.
Suppose that $[e_1,e_3] = \alpha e_1 + \beta e_2 + \gamma e_3$.
Then due to the Jacobi identity, we have 
\begin{multline*}
0 = [[e_1,e_2],e_3] + [[e_2,e_3],e_1] + [[e_3,e_1],e_2] 
  = \mu\lambda e_3 - \lambda(\alpha e_1 + \beta e_2 + \gamma e_3)   
  - \mu\alpha e_2 + \lambda \gamma e_3 \\
  = -\lambda\alpha e_1 - (\lambda\beta+\mu\alpha)e_2 + \mu\lambda e_3.
\end{multline*}
We get $\mu\lambda=0$, a~contradiction.

{\sc Case II}: $\Spec(R) = \{0,-1,t\}$, where $t\neq0,-1$. 
Let $e_1$ be an eigenvector corresponding to the eigenvalue~$t$.
Again, we write down,
$$
\langle e_1,e_2\rangle = (t+1)[e_1,e_2] = \mu e_2, \quad
\langle e_1,e_3\rangle = t[e_1,e_3] = \lambda e_3
$$
for some nonzero $\lambda,\mu\in \mathbb{C}$.
Suppose that $[e_2,e_3] = \alpha e_1 + \beta e_2 + \gamma e_3$.
Then the Jacobi identity implies $\beta = \gamma = 0$ and
$\lambda = -t\mu/(t+1)$. Considering $e_2/\sqrt{\alpha}$ and 
$e_3/\sqrt{\alpha}$ instead of $e_2$ and $e_3$, we get the following multiplication table:
$$
[e_1,e_2] = \Lambda e_2, \quad
[e_1,e_3] = -\Lambda e_3, \quad
[e_2,e_3] = e_1,
$$
where $\Lambda = \mu/(t+1)$. Also, $R(e_1) = te_1$, $R(e_2) = 0$ and $R(e_3) = -e_3$. 
Then $R$ is the RB-operator from the case c) up to conjugation with $\psi\in\Aut(\sl_2(\mathbb{C}))$
defined by the rule 
$\psi(e_1) = (\Lambda/2)h$, 
$\psi(e_2) = \sqrt{\Lambda/2}e$, 
$\psi(e_3) = \sqrt{\Lambda/2}f$. 
\end{proof}

\subsection{Conformal Lie algebras}

A (free) $\mathbb {C}[\partial]$-module $L$ is called a conformal algebra
if there is a $\lambda$-bracket on $L$,
$$
[\cdot_\lambda\cdot]\colon L\otimes L\to \mathbb{C}[\lambda]\otimes L,
$$
satisfying the identities,
$$
[\partial a_\lambda b] = -\lambda [a_\lambda b],\quad
[a_\lambda\partial b]=(\lambda +\partial)[a_\lambda b].
$$

A conformal algebra~$L$ is called a Lie conformal algebra if $L$ satisfies
the following conformal analogues of anticommutativity and the Jacobi identity:
$$
[a_\lambda b]=-[b_{-\lambda-\partial}a], \quad
[a_\lambda[b_\mu c]]
 - [b_\mu[a_\lambda c]]
 = [[a_\lambda b]_{\lambda+\mu}c].
$$

The Virasoro Lie conformal algebra $\mathrm{Vir}$ is defined as follows:
$$
\mathrm{Vir} = \mathbb{C}[\partial]L,\quad [L_{\lambda}L] = (\partial + 2\lambda)L.
$$

Given a Lie algebra~$\mathfrak{g}$, the current Lie conformal algebra $\Cur(\mathfrak{g})$
on the space $\mathbb{C}[\partial]\mathfrak{g}$ is defined by the formula
\begin{equation} \label{eq:ProductInCur(g)}
[f(\partial)a_\lambda g(\partial)b]
 = f(-\lambda)g(\lambda+\partial)[a,b], \quad a,b\in \mathfrak{g}.
\end{equation}

A left module $M$ over a Lie conformal algebra $L$ is a left
$\mathbb{C}[\partial]$-module endowed with a $\mathbb{C}$-linear map 
$(\cdot_\lambda\cdot)\colon C\otimes M \to M[\lambda]$ satisfying the identities
$$
\partial a_\lambda v = - \lambda a_\lambda v, \quad 
a_\lambda \partial v = (\partial+\lambda) a_\lambda v, \quad
[a_\lambda b]_{\lambda + \mu} v = a_\lambda (b_\mu v) - b_\mu (a_\lambda v)
$$
for all $a, b \in L$, $v \in M$.

Given a Lie conformal algebra~$L$, the space $L^{\otimes n}$ 
is a left $L$-module under the action
$$
a_\lambda (a_1 \otimes \dots \otimes a_n) 
 = \sum\limits_{i=1}^n a_1 \otimes \dots \otimes [a_\lambda a_i] \otimes \dots \otimes a_n,
$$
where $a, a_1, \dots a_n \in L$.

A~$\mathbb{C}[\partial]$-submodule $I$ of $L$ is called an ideal of $L$ if
$[I_\lambda L]\subset I$.
A Lie conformal algebra~$L$ is of finite type if $L$ is finitely-generated as $\mathbb{C}[\partial]$-module.
A Lie conformal algebra~$L$ is called simple if $[L_\lambda L]\neq(0)$
and there are only two ideals of $L$: $(0)$ and $L$.

Recall~\cite{SimpleConf} that every simple Lie conformal algebra of finite type
is isomorphic either to $\mathrm{Vir}$ or to $\Cur(\mathfrak{g})$
associated to a simple finite-dimensional Lie algebra $\mathfrak{g}$.

Given a Lie algebra~$\mathfrak{g}$ and an automorphism~$\varphi$ of $\mathfrak{g}$, we can extend it to 
an automorphism of $\Cur(\mathfrak{g})$ as a $\partial$-linear operator by the formula
$\varphi(f(\partial)a) = f(\partial)\varphi(a)$, $a\in \mathfrak{g}$.

\subsection{Conformal classical Yang---Baxter equation}

Let $L$ be a Lie conformal algebra and 
$r = \sum\nolimits a_i \otimes b_i \in L\otimes L$. Set 
$\partial_{\otimes 1} = \partial \otimes 1 \otimes 1$, 
$\partial_{\otimes 2} = 1 \otimes \partial \otimes 1$, 
$\partial_{\otimes 3} = 1 \otimes 1\otimes \partial$, and 
$\partial^{\otimes 3} =  \partial_{\otimes 1} + \partial_{\otimes 2} + \partial_{\otimes 3}$. 
The following equation,
\begin{multline} \label{CCYBE}
\llbracket r, r\rrbracket 
 := \sum\nolimits ([{a_i}_\lambda a_j] \otimes b_i \otimes b_j
\rvert_{\lambda = \partial_{\otimes 2}} 
- a_i \otimes [{a_j}_\lambda b_i] \otimes b_j \rvert_{\lambda 
= \partial_{\otimes 3}} \\
- a_i \otimes a_j \otimes [{b_j}_\lambda b_i] 
\rvert_{\lambda = \partial_{\otimes 2}}) = 
0~(\!\!\!\!\!\!\mod~\partial^{\otimes 3}),
\end{multline}
holding in $L^{\otimes 3}$, is called the~conformal classical Yang---Baxter equation (CCYBE)~\cite{Liberati}
and $r\in L\otimes L$ satisfying~\eqref{CCYBE} is called a solution to CCYBE.

The following equation fulfilled for all $a \in L$ is called the weak CCYBE:
\begin{equation} \label{weak-CCYBE}
a_\mu \llbracket r, r\rrbracket = 0\ (\!\!\!\!\!\!\mod \mu = - \partial^{\otimes 3}).
\end{equation}

A solution $r$ to CCYBE (or the weak one) is called $L$-invariant, 
if the following equality holds for every $a\in L$,
\begin{equation} \label{ConfInv}
a_{\lambda}(r + \tau(r))\rvert_{\lambda = - \partial^{\otimes 2}} = 0,
\end{equation}
where $\partial^{\otimes2} = \partial\otimes 1 + 1 \otimes \partial$
and $\tau\colon L\otimes L\to L\otimes L$ is defined as follows,
$\tau(a\otimes b) = b\otimes a$. 
A solution~$r$ to the (weak) CCYBE is called skew-symmetric if $r+\tau(r) = 0$.

In~\cite{Liberati}, J. Liberati proved that given a conformal Lie algebra~$L$ and $r\in L\otimes L$, the map
$\delta(a) = a_\lambda r|_{\lambda = -\partial^{\otimes2}}$
is a cocommutator of a conformal Lie bialgebra structure on~$L$ if and only if $r$ is an $L$-invariant solution to the weak CCYBE on~$L$.

\subsection{Rota---Baxter operators on conformal Lie algebras}

Given a Lie conformal algebra $L$, a $\partial$-linear map $R$ on $L$
is called a~Rota---Baxter operator (RB-operator, for short) of weight~$k\in\mathbb{C}$~\cite{Bai} if
\begin{equation}\label{RB}
[R(a)_\lambda R(b)]
 = R([R(a)_\lambda b] + [a_\lambda R(b)] + k[a_\lambda b])
\end{equation}
for all $a,b\in L$.

Given a Lie conformal algebra~$L$, let us call the operators
$R = 0$ and $R = -k\cdot\id$
as trivial Rota---Baxter operators of weight~$k$ on $L$.

In~\cite{Bai}, it was shown that there are only trivial RB-operators of weight~0 on $\mathrm{Vir}$.
It is easy to extend this result to the case of RB-operators of nonzero weight.

\begin{example}[\cite{Bai}]
Let $L = \mathbb{C}[\partial]a \oplus \mathbb{C}[\partial]b$ be a Lie conformal algebra of rank 2 with
the $\lambda$-bracket given by the formulas,
$$
[a_\lambda a] = (\partial + 2\lambda)a,\quad 
[a_\lambda b] = (\partial + \lambda)b, \quad
[b_\lambda b] = 0.
$$
Then any Rota--–Baxter operator of weight~0 on $L$ is one of the
following two forms:

(i) $R(a) = -\mu(a + b)$, $R(b) = \mu(a + b)$, where $\mu\in \mathbb{C}\setminus\{0\}$; 

(ii) $R(a) = g(\partial)b$, $R(b) = 0$, where $g(\partial) \in \mathbb{C}[\partial]$.
\end{example}

Given a Lie algebra $\mathfrak{g}$ and a $\partial$-linear map $R$ on $L = \Cur(\mathfrak{g})$,
define a linear operator $R_0$ on $\mathfrak{g}$ by the rule
$$
R_0(e_i) = R(e_i)|_{\partial = 0},
$$
where $\{e_i\mid i\in I\}$ is a linear basis of $\mathfrak{g}$. Note that the definition of $R_0$
does not depend on choice of a linear basis $\{e_i\}$.

\begin{lemma} 
Let $L = \Cur(\mathfrak{g})$ be a current conformal Lie algebra.

a) Suppose that $R$ is an RB-operator of weight~$k$ on $L$, then
$R_0$ is an RB-operator of weight~$k$ on $\mathfrak{g}$.

b) Suppose that $P$ is an RB-operator of weight~$k$ on $\mathfrak{g}$,
then the extension of $P$ on $L$ by the rule
$P(f(\partial)a) = f(\partial)P(a)$, where $a\in \mathfrak{g}$,
is an RB-operator of weight~$k$ on $L$.
\end{lemma} 

\begin{proof}
a) Given $a\in \mathfrak{g}$, we may present $R(a)$ as follows,
$$
R(a) = R_0(a) + \sum\limits_{k\geq1}\partial^k R_k(a),
$$
where $R_k(a)\in \mathfrak{g}$.
Then the identity~\eqref{RB} written down for $a,b\in \mathfrak{g}$
and $\lambda = \partial = 0$ gives exactly~\eqref{RBusual}.

b) It follows directly from the definitions of RB-operators and
$\lambda$-product in the current algebra.
\end{proof}

Analogously to Lemma~\ref{lem:phi}, a conjugation $R^{(\psi)}$
of an RB-operator~$R$ on a Lie conformal algebra~$L$
with an automorphism $\psi$ of~$L$ is again an RB-operator of the same weight~\cite{Zhao2021}.

\begin{corollary} \label{coro}
Let $L = \Cur(\mathfrak{g})$ be a current conformal Lie algebra,
let $R$ be an RB-operator of weight~$\lambda$ on $L$,
and let $\psi_0$ be an automorphism of $\mathfrak{g}$.
Then $(R^{(\psi)})_0 = (R_0)^{(\psi_0)}$,
where $\psi$ is a $\partial$-invariant extension of $\psi_0$
to an automorphism of $L$.
\end{corollary}

\section{Rota---Baxter operators of weight 0 on $\Cur(\sl_2(\mathbb{C}))$}

Let us give a general example of RB-operators of weight 0 on the current algebra $\Cur(\mathfrak{g})$ of a~finite-dimensional semisimple Lie algebra~$\mathfrak{g}$. 
We will see further that this construction implies one of two RB-operators on $\Cur(\sl_2(\mathbb{C}))$ which are not $\partial$-linear extensions of RB-operators on $\sl_2(\mathbb{C})$.

\begin{example}
Let $\mathfrak{g}$ be a finite-dimensional semisimple Lie algebra over $\mathbb{C}$.
Let a linear operator $R$ acts on $L = \Cur(\mathfrak{g})$ as follows,
$R(\mathfrak{h})\subset \mathbb{C}[\partial]\mathfrak{h}$ for a Cartan subalgebra $\mathfrak{h}$ of $\mathfrak{g}$, and $R$ maps all other weighted subspaces to zero.
Then $R$ is a Rota---Baxter operator of weight~0 on $L$.
\end{example}

\subsection{Classification}

Let $R$ be an RB-operator of weight~0 on $\Cur(\sl_2(\mathbb{C}))$.
Introduce
\begin{gather}
R(e) = a_e(\partial)e + a_f(\partial)f + a_h(\partial)h, \nonumber \\
R(f) = b_e(\partial)e + b_f(\partial)f + b_h(\partial)h, \label{ROnBasis} \\
R(h) = c_e(\partial)e + c_f(\partial)f + c_h(\partial)h. \nonumber
\end{gather}

Now, we compute by~\eqref{eq:ProductInCur(g)}
\begin{multline*}
[R(h)_\lambda R(h)]
 = [(c_e(\partial)e + c_f(\partial)f + c_h(\partial)h)_\lambda(c_e(\partial)e + c_f(\partial)f + c_h(\partial)h)] \\
 = 2(c_e(\lambda+\partial)c_h(-\lambda)-c_e(-\lambda)c_h(\lambda+\partial))e \allowdisplaybreaks \\
 - 2(c_f(\lambda+\partial)c_h(-\lambda)-c_f(-\lambda)c_h(\lambda+\partial))f \\
 + (c_e(-\lambda)c_f(\lambda+\partial)-c_e(\lambda+\partial)c_f(-\lambda))h;
\end{multline*}

\vspace{-0.75cm}
\begin{multline*}
R([R(h)_\lambda h] + [h_\lambda R(h)])
 {=} R([(c_e(\partial)e + c_f(\partial)f + c_h(\partial)h)_\lambda h] + [h_\lambda (c_e(\partial)e + c_f(\partial)f + c_h(\partial)h)]) \\
 = 2R( (c_e(\lambda+\partial) - c_e(-\lambda))e + (c_f(-\lambda) - c_f(\lambda+\partial))f ) \allowdisplaybreaks \\
 = 2(c_e(\lambda+\partial) - c_e(-\lambda))(a_e(\partial)e + a_f(\partial)f + a_h(\partial)h) \\
 + 2(c_f(-\lambda) - c_f(\lambda+\partial))(b_e(\partial)e + b_f(\partial)f + b_h(\partial)h).
\end{multline*}
Thus, we obtain the following equations,
\begin{multline*}
c_e(\lambda+\partial)c_h(-\lambda) - c_e(-\lambda)c_h(\lambda+\partial) \\
 = a_e(\partial)(c_e(\lambda+\partial)-c_e(-\lambda)) + b_e(\partial)(c_f(-\lambda)-c_f(\lambda+\partial)),
\end{multline*}

\vspace{-1cm}
$$
c_f(-\lambda)c_h(\lambda+\partial) - c_f(\lambda+\partial)c_h(-\lambda)
 = a_f(\partial)(c_e(\lambda+\partial)-c_e(-\lambda)) + b_f(\partial)(c_f(-\lambda)-c_f(\lambda+\partial)),
$$

\vspace{-1cm}
\begin{multline*}
c_e(-\lambda)c_f(\lambda+\partial) - c_e(\lambda+\partial)c_f(-\lambda) \\
 = 2a_h(\partial)(c_e(\lambda+\partial)-c_e(-\lambda)) + 2b_h(\partial)(c_f(-\lambda)-c_f(\lambda+\partial)).
\end{multline*}
Analogously we get from~\eqref{RB} applied for
$[R(e)_\lambda R(e)]$ and $[R(f)_\lambda R(f)]$ the equalities
\begin{multline*}
2a_e(\lambda+\partial)a_h(-\lambda) - 2a_e(-\lambda)a_h(\lambda+\partial) \\
 = c_e(\partial)(a_f(\lambda+\partial)-a_f(-\lambda)) + 2a_e(\partial)(a_h(-\lambda)-a_h(\lambda+\partial)),
\end{multline*}

\vspace{-1cm}
\begin{multline*}
2a_f(-\lambda)a_h(\lambda+\partial) - 2a_f(\lambda+\partial)a_h(-\lambda) \\
 = c_f(\partial)(a_f(\lambda+\partial)-a_f(-\lambda)) + 2a_f(\partial)(a_h(-\lambda)-a_h(\lambda+\partial)),
\end{multline*}

\vspace{-1cm}
$$
a_e(-\lambda)a_f(\lambda+\partial) - a_e(\lambda+\partial)a_f(-\lambda)
 {=} c_h(\partial)(a_f(\lambda+\partial)-a_f(-\lambda)) + 2a_h(\partial)(a_h(-\lambda)-a_h(\lambda+\partial)).
$$

\vspace{-1cm}
\begin{multline*}
2b_e(\lambda+\partial)b_h(-\lambda) - 2b_e(-\lambda)b_h(\lambda+\partial) \\
 = c_e(\partial)(b_e(-\lambda)-b_e(\lambda+\partial)) + 2b_e(\partial)(b_h(\lambda+\partial)-b_h(-\lambda)),
\end{multline*}

\vspace{-1cm}
\begin{multline*}
2b_f(-\lambda)b_h(\lambda+\partial) - 2b_f(\lambda+\partial)b_h(-\lambda) \\
 = c_f(\partial)(b_e(-\lambda)-b_e(\lambda+\partial)) + 2b_f(\partial)(b_h(\lambda+\partial)-b_h(-\lambda)),
\end{multline*}

\vspace{-1cm}
\begin{multline*}
b_e(-\lambda)b_f(\lambda+\partial) - b_e(\lambda+\partial)b_f(-\lambda) \\
 = c_h(\partial)(b_e(-\lambda)-b_e(\lambda+\partial)) + 2b_h(\partial)(b_h(\lambda+\partial)-b_h(-\lambda)).
\end{multline*}
Let us compare
\begin{multline*}
[R(e)_\lambda R(f)]
 = [(a_e(\partial)e + a_f(\partial)f + a_h(\partial)h)_\lambda(b_e(\partial)e + b_f(\partial)f + b_h(\partial)h)] \\
 = 2(a_h(-\lambda)b_e(\lambda+\partial)-a_e(-\lambda)b_h(\lambda+\partial))e \\
 - 2(a_h(-\lambda)b_f(\lambda+\partial)-a_f(-\lambda)b_h(\lambda+\partial))f \\
 + (a_e(-\lambda)b_f(\lambda+\partial)-a_f(-\lambda)b_e(\lambda+\partial))h;
\end{multline*}

\vspace{-1cm}
\begin{multline*}
R([R(e)_\lambda f] + [e_\lambda R(f)])
 {=} R([(a_e(\partial)e + a_f(\partial)f + a_h(\partial)h)_\lambda f] + [e_\lambda (b_e(\partial)e + b_f(\partial)f + b_h(\partial)h)]) \\
 = R( -2b_h(\lambda+\partial)e - 2a_h(-\lambda)f + (a_e(-\lambda)+b_f(\lambda+\partial))h ) \\
 = -2b_h(\lambda+\partial)(a_e(\partial)e + a_f(\partial)f + a_h(\partial)h)
 - 2a_h(-\lambda)(b_e(\partial)e + b_f(\partial)f + b_h(\partial)h) \\
 + (a_e(-\lambda)+b_f(\lambda+\partial))(c_e(\partial)e + c_f(\partial)f + c_h(\partial)h).
\end{multline*}
We get new three identities:
\begin{multline*}
2(a_h(-\lambda)b_e(\lambda+\partial)-a_e(-\lambda)b_h(\lambda+\partial)) \\
 = -2b_h(\lambda+\partial)a_e(\partial) - 2a_h(-\lambda)b_e(\partial)  + (a_e(-\lambda)+b_f(\lambda+\partial))c_e(\partial),
\end{multline*}

\vspace{-1cm}
\begin{multline*}
- 2(a_h(-\lambda)b_f(\lambda+\partial)-a_f(-\lambda)b_h(\lambda+\partial)) \\
 = -2b_h(\lambda+\partial)a_f(\partial) - 2a_h(-\lambda)b_f(\partial)  + (a_e(-\lambda)+b_f(\lambda+\partial))c_f(\partial),
\end{multline*}

\vspace{-1cm}
\begin{multline*}
a_e(-\lambda)b_f(\lambda+\partial)-a_f(-\lambda)b_e(\lambda+\partial) \\
 = -2b_h(\lambda+\partial)a_h(\partial) - 2a_h(-\lambda)b_h(\partial)  + (a_e(-\lambda)+b_f(\lambda+\partial))c_h(\partial).
\end{multline*}
Further, we write down,
\begin{multline*}
[R(e)_\lambda R(h)]
 = [(a_e(\partial)e + a_f(\partial)f + a_h(\partial)h)_\lambda(c_e(\partial)e + c_f(\partial)f + c_h(\partial)h)] \\
 = 2(a_h(-\lambda)c_e(\lambda+\partial)-a_e(-\lambda)c_h(\lambda+\partial))e \\
 - 2(a_h(-\lambda)c_f(\lambda+\partial)-a_f(-\lambda)c_h(\lambda+\partial))f \\
 + (a_e(-\lambda)c_f(\lambda+\partial)-a_f(-\lambda)c_e(\lambda+\partial))h;
\end{multline*}

\vspace{-1cm}
\begin{multline*}
R([R(e)_\lambda h] + [e_\lambda R(h)])
 {=} R([(a_e(\partial)e + a_f(\partial)f + a_h(\partial)h)_\lambda h] + [e_\lambda (c_e(\partial)e + c_f(\partial)f + c_h(\partial)h)]) \\
 = R( -2( a_e(-\lambda) + c_h(\lambda+\partial) )e + 2a_f(-\lambda)f + c_f(\lambda+\partial)h ) \\
 = -2( a_e(-\lambda) + c_h(\lambda+\partial) )(a_e(\partial)e + a_f(\partial)f + a_h(\partial)h) \\
 + 2a_f(-\lambda)(b_e(\partial)e + b_f(\partial)f + b_h(\partial)h)
 + c_f(\lambda+\partial)(c_e(\partial)e + c_f(\partial)f + c_h(\partial)h).
\end{multline*}
We get new three identities:
\begin{multline*}
2(a_h(-\lambda)c_e(\lambda+\partial)-a_e(-\lambda)c_h(\lambda+\partial)) \\
 = -2( a_e(-\lambda) + c_h(\lambda+\partial) )a_e(\partial) + 2a_f(-\lambda)b_e(\partial)  + c_f(\lambda+\partial)c_e(\partial),
\end{multline*}

\vspace{-1cm}
\begin{multline*}
- 2(a_h(-\lambda)c_f(\lambda+\partial)-a_f(-\lambda)c_h(\lambda+\partial)) \\
 = -2( a_e(-\lambda) + c_h(\lambda+\partial) )a_f(\partial) + 2a_f(-\lambda)b_f(\partial)  + c_f(\lambda+\partial)c_f(\partial),
\end{multline*}

\vspace{-1cm}
\begin{multline*}
a_e(-\lambda)c_f(\lambda+\partial)-a_f(-\lambda)c_e(\lambda+\partial) \\
 = -2( a_e(-\lambda) + c_h(\lambda+\partial) )a_h(\partial) + 2a_f(-\lambda)b_h(\partial)  + c_f(\lambda+\partial)c_h(\partial).
\end{multline*}
Finally, we have,
\begin{multline*}
[R(f)_\lambda R(h)]
 = [(b_e(\partial)e + b_f(\partial)f + b_h(\partial)h)_\lambda(c_e(\partial)e + c_f(\partial)f + c_h(\partial)h)] \\
 = 2(b_h(-\lambda)c_e(\lambda+\partial)-b_e(-\lambda)c_h(\lambda+\partial))e \\
 - 2(b_h(-\lambda)c_f(\lambda+\partial)-b_f(-\lambda)c_h(\lambda+\partial))f \\
 + (b_e(-\lambda)c_f(\lambda+\partial)-b_f(-\lambda)c_e(\lambda+\partial))h;
\end{multline*}

\vspace{-1cm}
\begin{multline*}
R([R(f)_\lambda h] + [f_\lambda R(h)])
 {=} R([(b_e(\partial)e + b_f(\partial)f + b_h(\partial)h)_\lambda h] + [f_\lambda (c_e(\partial)e + c_f(\partial)f + c_h(\partial)h)]) \\
 = R( -2b_e(-\lambda)e + 2(b_f(-\lambda)+c_h(\lambda+\partial))f - c_e(\lambda+\partial)h ) \\
 = -2b_e(-\lambda)(a_e(\partial)e + a_f(\partial)f + a_h(\partial)h) \\
 + 2(b_f(-\lambda)+c_h(\lambda+\partial))(b_e(\partial)e + b_f(\partial)f + b_h(\partial)h)
 - c_e(\lambda+\partial)(c_e(\partial)e + c_f(\partial)f + c_h(\partial)h).
\end{multline*}
The last three identities are
\begin{multline*}
2(b_h(-\lambda)c_e(\lambda+\partial)-b_e(-\lambda)c_h(\lambda+\partial)) \\
 = -2b_e(-\lambda)a_e(\partial) + 2(b_f(-\lambda)+c_h(\lambda+\partial))b_e(\partial)  - c_e(\lambda+\partial)c_e(\partial),
\end{multline*}

\vspace{-1cm}
\begin{multline*}
- 2(b_h(-\lambda)c_f(\lambda+\partial)-b_f(-\lambda)c_h(\lambda+\partial)) \\
 = -2b_e(-\lambda)a_f(\partial) + 2(b_f(-\lambda)+c_h(\lambda+\partial))b_f(\partial)  - c_e(\lambda+\partial)c_f(\partial),
\end{multline*}

\vspace{-1cm}
\begin{multline*}
b_e(-\lambda)c_f(\lambda+\partial)-b_f(-\lambda)c_e(\lambda+\partial) \\
 = -2b_e(-\lambda)a_h(\partial) + 2(b_f(-\lambda)+c_h(\lambda+\partial))b_h(\partial)  - c_e(\lambda+\partial)c_h(\partial).
\end{multline*}
Altogether, we have 18 equations:
\begin{multline}\label{1}
c_e(\lambda+\partial)c_h(-\lambda) - c_e(-\lambda)c_h(\lambda+\partial) \\
 = a_e(\partial)(c_e(\lambda+\partial)-c_e(-\lambda)) + b_e(\partial)(c_f(-\lambda)-c_f(\lambda+\partial)),
\end{multline}

\vspace{-0.9cm}
\begin{multline}\label{2}
c_f(-\lambda)c_h(\lambda+\partial) - c_f(\lambda+\partial)c_h(-\lambda) \\
 = a_f(\partial)(c_e(\lambda+\partial)-c_e(-\lambda)) + b_f(\partial)(c_f(-\lambda)-c_f(\lambda+\partial)),
\end{multline}

\vspace{-0.9cm}
\begin{multline}\label{3}
c_e(-\lambda)c_f(\lambda+\partial) - c_e(\lambda+\partial)c_f(-\lambda) \\
 = 2a_h(\partial)(c_e(\lambda+\partial)-c_e(-\lambda)) + 2b_h(\partial)(c_f(-\lambda)-c_f(\lambda+\partial)),
\end{multline}

\vspace{-0.9cm}
\begin{multline}\label{4}
2a_e(\lambda+\partial)a_h(-\lambda) - 2a_e(-\lambda)a_h(\lambda+\partial) \\
 = c_e(\partial)(a_f(\lambda+\partial)-a_f(-\lambda)) + 2a_e(\partial)(a_h(-\lambda)-a_h(\lambda+\partial)),
\end{multline}

\vspace{-0.9cm}
\begin{multline}\label{5}
2a_f(-\lambda)a_h(\lambda+\partial) - 2a_f(\lambda+\partial)a_h(-\lambda) \\
 = c_f(\partial)(a_f(\lambda+\partial)-a_f(-\lambda)) + 2a_f(\partial)(a_h(-\lambda)-a_h(\lambda+\partial)),
\end{multline}

\vspace{-0.9cm}
\begin{multline}\label{6}
a_e(-\lambda)a_f(\lambda+\partial) - a_e(\lambda+\partial)a_f(-\lambda) \\
 = c_h(\partial)(a_f(\lambda+\partial)-a_f(-\lambda)) + 2a_h(\partial)(a_h(-\lambda)-a_h(\lambda+\partial)),
\end{multline}

\vspace{-0.9cm}
\begin{multline}\label{7}
2b_e(\lambda+\partial)b_h(-\lambda) - 2b_e(-\lambda)b_h(\lambda+\partial) \\
 = c_e(\partial)(b_e(-\lambda)-b_e(\lambda+\partial)) + 2b_e(\partial)(b_h(\lambda+\partial)-b_h(-\lambda)),
\end{multline}

\vspace{-0.9cm}
\begin{multline}\label{8}
2b_f(-\lambda)b_h(\lambda+\partial) - 2b_f(\lambda+\partial)b_h(-\lambda) \\
 = c_f(\partial)(b_e(-\lambda)-b_e(\lambda+\partial)) + 2b_f(\partial)(b_h(\lambda+\partial)-b_h(-\lambda)),
\end{multline}

\vspace{-0.9cm}
\begin{multline}\label{9}
b_e(-\lambda)b_f(\lambda+\partial) - b_e(\lambda+\partial)b_f(-\lambda) \\
 = c_h(\partial)(b_e(-\lambda)-b_e(\lambda+\partial)) + 2b_h(\partial)(b_h(\lambda+\partial)-b_h(-\lambda)),
\end{multline}

\vspace{-0.9cm}
\begin{multline}\label{10}
2(a_h(-\lambda)b_e(\lambda+\partial)-a_e(-\lambda)b_h(\lambda+\partial)) \\
 = -2b_h(\lambda+\partial)a_e(\partial) - 2a_h(-\lambda)b_e(\partial) + (a_e(-\lambda)+b_f(\lambda+\partial))c_e(\partial),
\end{multline}

\vspace{-0.9cm}
\begin{multline}\label{11}
- 2(a_h(-\lambda)b_f(\lambda+\partial)-a_f(-\lambda)b_h(\lambda+\partial)) \\
 = -2b_h(\lambda+\partial)a_f(\partial) - 2a_h(-\lambda)b_f(\partial) + (a_e(-\lambda)+b_f(\lambda+\partial))c_f(\partial),
\end{multline}

\vspace{-0.9cm}
\begin{multline}\label{12}
a_e(-\lambda)b_f(\lambda+\partial)-a_f(-\lambda)b_e(\lambda+\partial) \\
 = -2b_h(\lambda+\partial)a_h(\partial) - 2a_h(-\lambda)b_h(\partial) + (a_e(-\lambda)+b_f(\lambda+\partial))c_h(\partial),
\end{multline}

\vspace{-0.9cm}
\begin{multline}\label{13}
2(a_h(-\lambda)c_e(\lambda+\partial)-a_e(-\lambda)c_h(\lambda+\partial)) \\
 = -2( a_e(-\lambda) + c_h(\lambda+\partial) )a_e(\partial) + 2a_f(-\lambda)b_e(\partial) + c_f(\lambda+\partial)c_e(\partial),
\end{multline}

\vspace{-0.9cm}
\begin{multline}\label{14}
- 2(a_h(-\lambda)c_f(\lambda+\partial)-a_f(-\lambda)c_h(\lambda+\partial)) \\
 = -2( a_e(-\lambda) + c_h(\lambda+\partial) )a_f(\partial) + 2a_f(-\lambda)b_f(\partial) + c_f(\lambda+\partial)c_f(\partial),
\end{multline}

\vspace{-0.9cm}
\begin{multline}\label{15}
a_e(-\lambda)c_f(\lambda+\partial)-a_f(-\lambda)c_e(\lambda+\partial) \\
 = -2( a_e(-\lambda) + c_h(\lambda+\partial) )a_h(\partial) + 2a_f(-\lambda)b_h(\partial) + c_f(\lambda+\partial)c_h(\partial),
\end{multline}

\vspace{-0.9cm}
\begin{multline}\label{16}
2(b_h(-\lambda)c_e(\lambda+\partial)-b_e(-\lambda)c_h(\lambda+\partial)) \\
 = -2b_e(-\lambda)a_e(\partial) + 2(b_f(-\lambda)+c_h(\lambda+\partial))b_e(\partial) - c_e(\lambda+\partial)c_e(\partial),
\end{multline}

\vspace{-0.9cm}
\begin{multline}\label{17}
- 2(b_h(-\lambda)c_f(\lambda+\partial)-b_f(-\lambda)c_h(\lambda+\partial)) \\
 = -2b_e(-\lambda)a_f(\partial) + 2(b_f(-\lambda)+c_h(\lambda+\partial))b_f(\partial) - c_e(\lambda+\partial)c_f(\partial),
\end{multline}

\vspace{-0.9cm}
\begin{multline}\label{18}
b_e(-\lambda)c_f(\lambda+\partial)-b_f(-\lambda)c_e(\lambda+\partial) \\
 = -2b_e(-\lambda)a_h(\partial) + 2(b_f(-\lambda)+c_h(\lambda+\partial))b_h(\partial) - c_e(\lambda+\partial)c_h(\partial).
\end{multline}

We can extend any automorphism of $\sl_2(\mathbb{C})$
to an automorphism of $L = \Cur(\sl_2(\mathbb{C}))$.
Hence, by Proposition~\ref{prop:RBOnsl2-0} and by Corollary~\ref{coro} we may assume that
\begin{equation}\label{R0-sl2-zero}
a_e(0) = a_f(0) = a_h(0) = 0, \quad
b_f(0) = 0, \quad c_f(0) = 0.
\end{equation}
Let us derive consequences of these identities.
Consider~\eqref{10},~\eqref{4},~\eqref{8},~\eqref{6},~\eqref{9},~\eqref{14},
and~\eqref{16} respectively with $\lambda+\partial = 0$:
\begin{gather}
2a_h(\partial)(b_e(\partial)+b_e(0)) = c_e(\partial)a_e(\partial), \label{ah-ce1}  \\
2a_h(\partial)a_e(\partial) = c_e(\partial)a_f(\partial), \label{ah-ce2} \\
2b_h(\partial)b_f(\partial) = c_f(\partial)(b_e(\partial)-b_e(0)), \label{2bhbf=cfbe} \\
2a_h^2(\partial) = a_f(\partial)c_h(\partial), \label{ah^2} \allowdisplaybreaks \\
2b_h(\partial)(b_h(\partial)-b_h(0))
 = b_e(\partial)c_h(\partial) + b_e(0)(b_f(\partial) - c_h(\partial)), \label{bh^2} \\
0 = a_f(\partial)(a_e(\partial) - b_f(\partial) + 2c_h(0)), \label{cf-af} \\
c_e(0)(c_e(\partial) + 2b_h(\partial))
 = 2b_e(\partial)(b_f(\partial) - a_e(\partial) + 2c_h(0)). \label{ce-be}
\end{gather}
Take~$\lambda = 0$ in~\eqref{5},~\eqref{7}, and~\eqref{11} to get
\begin{gather}
0 = a_f(\partial)(c_f(\partial) - 2a_h(\partial)), \label{af-af} \\
b_e(0)(c_e(\partial) + 2b_h(\partial))
 = b_e(\partial)(c_e(\partial) - 2b_h(\partial) + 4b_h(0)), \label{be-be} \\
2b_h(\partial)a_f(\partial) = c_f(\partial)b_f(\partial). \label{2bhaf=cfbf}
\end{gather}
Substituting $\lambda + \partial = 0$ in~\eqref{13} and~\eqref{17}, we obtain
\begin{gather}
a_e^2(\partial) - a_f(\partial)b_e(\partial) + a_h(\partial)c_e(0) = 0, \label{ae^2-af*be} \\
2b_f^2(\partial) - 2a_f(\partial)b_e(\partial) - c_e(0)c_f(\partial) = 0. \label{bf^2-af*be}
\end{gather}

{\sc Case I}: $a_f(\partial) = 0$.
By~\eqref{ah^2}, we get $a_h(\partial) = 0$.
By~\eqref{ae^2-af*be}, we have $a_e(\partial) = 0$.
By~\eqref{2bhaf=cfbf} and~\eqref{bf^2-af*be}, we obtain
$b_f(\partial) = 0$.
By~\eqref{14} considered at $\lambda = 0$, we have $c_f(\partial) = 0$.
We have the following system on $b_e(\partial),b_h(\partial),c_e(\partial),c_h(\partial)$
remaining from the equations~\eqref{1},~\eqref{7},~\eqref{9},~\eqref{16},~\eqref{18}:
\begin{equation}\label{af=0-1new}
c_e(\lambda+\partial)c_h(-\lambda) - c_e(-\lambda)c_h(\lambda+\partial) = 0,
\end{equation}

\vspace{-1.1cm}
\begin{multline}\label{af=0-7new}
2b_e(\lambda+\partial)b_h(-\lambda) - 2b_e(-\lambda)b_h(\lambda+\partial) \\
 = c_e(\partial)(  b_e(-\lambda) - b_e(\lambda+\partial) ) + 2b_e(\partial)(b_h(\lambda+\partial)-b_h(-\lambda)),
\end{multline}

\vspace{-1cm}
\begin{gather}
c_h(\partial)(b_e(-\lambda) - b_e(\lambda+\partial)) + 2b_h(\partial)(b_h(\lambda+\partial) - b_h(-\lambda)) = 0, \label{af=0-9new} \\
2(b_h(-\lambda)c_e(\lambda+\partial) - b_e(-\lambda)c_h(\lambda+\partial))
 = 2c_h(\lambda+\partial)b_e(\partial) - c_e(\lambda+\partial)c_e(\partial), \label{af=0-16new} \\
2c_h(\lambda+\partial) b_h(\partial) - c_e(\lambda+\partial)c_h(\partial) = 0. \label{af=0-18new}
\end{gather}

{\sc Case IA}: $c_h(\partial) = 0$.
Then $b_h(\partial) = b_h(0)\in \mathbb{C}$ by~\eqref{af=0-9new}.
By~\eqref{af=0-16new}, we have $c_e(\partial) = c_e(0)\in \mathbb{C}$
and, moreover, $c_e(0) = 0$ or $c_e(0) = - 2b_h(0)$.
Excluding the case when $R = R_0$, we get
$$
R_1(e) = 0, \quad
R_1(f) = b_e(\partial)e + \alpha h,\quad
R_1(h) = -2\alpha e
$$
for some $\alpha\in \mathbb{C}$ and nonconstant $b_e(\partial)$.

{\sc Case IB}: $c_h(\partial) \neq 0$.
By~\eqref{af=0-1new}, we have
$$
\frac{c_e(-\lambda)}{c_h(-\lambda)}
 = \frac{c_e(\lambda+\partial)}{c_h(\lambda+\partial)},
$$
it means that $c_e(\partial) = \alpha c_h(\partial)$ for some $\alpha\in \mathbb{C}$.
Thus, by~\eqref{af=0-18new}, we get
$b_h(\partial) = (\alpha/2)c_h(\partial)$.
By~\eqref{af=0-16new},
$b_e(\partial) = (\alpha^2/2)c_h(\partial)$.
Therefore, we get the RB-operator:
$$
R_2(e) = 0, \quad
R_2(f) = (\alpha/2)q(\partial)( \alpha e+h ), \quad
R_2(h) = q(\partial)( \alpha e+h ).
$$

{\sc Case II}: $a_f(\partial)\neq0$ and $b_e(\partial) = 0$.
By~\eqref{ah-ce1}, we have $a_e(\partial)c_e(\partial) = 0$.
Thus, by~\eqref{ah-ce2}, $a_f(\partial)c_e(\partial) = 0$.
Since $a_f(\partial)\neq0$, we conclude that $c_e(\partial) = 0$.

By~\eqref{ae^2-af*be} and~\eqref{bf^2-af*be}, we get
$a_e(\partial) = b_f(\partial) = 0$.
The equality~\eqref{2bhaf=cfbf} implies that $b_h(\partial) = 0$.
Now we may conjugate $R$ with an automorphism
$\varphi$ such that $\varphi(e) = f$, $\varphi(f) = e$, and $\varphi(h) = -h$.
So, we move to the case I.

{\sc Case III}: $a_f(\partial) \neq 0$ and $b_e(\partial) \neq 0$.
By~\eqref{cf-af} and by~\eqref{af-af}, we have
\begin{gather}
c_f(\partial) = 2a_h(\partial), \label{cf=2ah} \\
a_e(\partial) = b_f(\partial) - 2c_h(0). \label{ae=bf}
\end{gather}
The last equality implies $c_h(0) = (b_f(0) - a_e(0))/2 = 0$.
So, $a_e(\partial) = b_f(\partial)$.

From~\eqref{2bhaf=cfbf}, \eqref{cf=2ah}, \eqref{ah-ce2}
and the relation $a_e(\partial) = b_f(\partial)$, we conclude
$$
2b_h(\partial)a_f(\partial)
 = c_f(\partial)b_f(\partial)
 = 2a_h(\partial)b_f(\partial)
 = 2a_h(\partial)a_e(\partial)
 = c_e(\partial)a_f(\partial),
$$
i.\,e.,
\begin{equation}\label{ce=2bh}
c_e(\partial) = 2b_h(\partial).
\end{equation}
The equality~\eqref{ce-be} considered at $\partial = 0$ implies $c_e(0) = b_h(0) = 0$.

Therefore, by~\eqref{ae^2-af*be}, we have
$a_e^2(\partial) = b_f^2(\partial) = a_f(\partial)b_e(\partial)\neq0$.
Subtracting~\eqref{2bhbf=cfbe} from~\eqref{ah-ce1} we obtain, by~\eqref{ce=2bh},
$$
a_h(\partial)b_e(0) = 0.
$$

{\sc Case IIIa}: $a_h(\partial) = 0$. 
By~\eqref{cf=2ah} and~\eqref{2bhaf=cfbf}, $c_f(\partial) = b_h(\partial) = 0$.
Then by~\eqref{ah-ce2}, we derive $c_e(\partial) = 0$.
By~\eqref{6} considered at $\lambda = 0$, we get $c_h(\partial)a_f(\partial) = 0$, so,
$c_h(\partial) = 0$. 
By~\eqref{12} computed at $\lambda+\partial = 0$, we have $b_e(0)a_f(\partial) = 0$,
so, $b_e(0) = 0$. 
Since $a_e(\partial)\neq0$, we apply~\eqref{6}
to get the equality $a_f(\partial) = \alpha a_e(\partial)$ with some nonzero $\alpha\in\mathbb{C}$.
Thus, we obtain the RB-operator~$R$ such that
$$
R_3(e) = q(\partial)(e+\alpha f),\quad
R_3(f) = q(\partial)((1/\alpha)e+f),\quad
R_3(h) = 0.
$$

{\sc Case IIIb}: $a_h(\partial) \neq 0$. 
Then $b_e(0) = 0$ and $b_h(\partial)\neq0$ by~\eqref{ce=2bh}.
By~\eqref{ah-ce1} and~\eqref{ah-ce2}, we get
$$
\frac{b_h(\partial)}{a_h(\partial)}
 = \frac{b_f(\partial)}{a_f(\partial)}
 = \frac{b_e(\partial)}{a_e(\partial)} = \varphi(\partial),
$$
where $\varphi(\partial)$ is a~rational function on $\partial$.

By~\eqref{ce=2bh},~\eqref{2bhaf=cfbf}, and~\eqref{12} with $\lambda = 0$, we get
$$
\frac{2b_h(\partial)}{b_f(\partial)}
 = \frac{c_h(\partial)}{a_h(\partial)}
 = \frac{c_f(\partial)}{a_f(\partial)}
 = \frac{c_e(\partial)}{a_e(\partial)}
 = 2\psi(\partial),
$$
for another rational function~$\psi(\partial)$.
By $c_e(\partial) = 2b_h(\partial)$, $c_f(\partial) = 2a_h(\partial)$,
and $b_f(\partial) = a_e(\partial)$,
we get the following form of the matrix $[R]$ of $R$ in the basis $e,f,h$:
$$
[R] = a_f(\partial)\begin{pmatrix}
\varphi(\partial) & 1 & \psi(\partial) \\
\varphi^2(\partial) & \varphi(\partial) & \varphi(\partial)\psi(\partial) \\
2\varphi(\partial)\psi(\partial) & 2\psi(\partial) & 2\psi^2(\partial)
\end{pmatrix}.
$$
Denote $a_f(\partial) = q(\partial)$.
The right hand-side of~\eqref{2} multiplied by $\varphi(\partial)$
and $2\psi(\partial)$ coincides
with the right hand-side of~\eqref{1} and \eqref{3} respectively.
Thus, we get the equalities
\begin{multline*}
q(-\lambda)q(\lambda+\partial)\psi(\lambda+\partial)\psi(-\lambda)
 ( \varphi(\lambda+\partial)\psi(-\lambda)
 - \varphi(-\lambda)\psi(\lambda+\partial) ) \\
  = q(-\lambda)q(\lambda+\partial)\psi(\lambda+\partial)\psi(-\lambda)
 \varphi(\partial)( \psi(\lambda+\partial) - \psi(-\lambda) ),
\end{multline*}
\begin{multline*}
q(-\lambda)q(\lambda+\partial)\psi(\lambda+\partial)\psi(-\lambda)
 ( \varphi(-\lambda) - \varphi(\lambda+\partial) ) \\
 = 2q(-\lambda)q(\lambda+\partial)\psi(\lambda+\partial)\psi(-\lambda)
 \psi(\partial)( \psi(\lambda+\partial) - \psi(-\lambda) ).
\end{multline*}
Simplifying both relations, we obtain
\begin{gather*}
\varphi(\lambda+\partial)\psi(-\lambda)
 - \varphi(-\lambda)\psi(\lambda+\partial)
  = \varphi(\partial)( \psi(\lambda+\partial) - \psi(-\lambda) ), \\
( \varphi(-\lambda) - \varphi(\lambda+\partial) )
  = 2\psi(\partial)( \psi(\lambda+\partial) - \psi(-\lambda) ).
\end{gather*}
Thus, we get
$$
\varphi(-\lambda)(\psi(-\lambda) - \psi(\lambda+\partial))
 = (2\psi(\partial)\psi(-\lambda)+\varphi(\partial))( \psi(\lambda+\partial) - \psi(-\lambda) ).
$$
Suppose that $\psi(\partial)$ is not a constant. Then
$2\psi(\partial)\psi(-\lambda)+\varphi(\partial) = - \varphi(-\lambda)$.
Setting $-\lambda = \partial$, we get $\psi^2(\partial) = -\varphi(\partial)$.
So, $\psi(\partial)(2\psi(0)-\psi(\partial)) = -\varphi(0)$. 
It means that both rational functions $\psi(\partial)$ and $\varphi(\partial)$ are constant, a~contradiction.

Thus $\psi(\partial) = \beta\in \mathbb{C}\setminus\{0\}$, and it is easy to show that 
$\varphi(\partial) = \alpha\in \mathbb{C}\setminus\{0\}$.
Summarizing, we get
$$
R_4(e) = q(\partial)(\alpha e + f + \beta h),\quad
R_4(\alpha e-f) = R_4(2\beta e - h) = 0,\quad \alpha,\beta\neq0.
$$

\begin{theorem} 
Up to conjugation with an automorphism of $\Cur(\sl_2(\mathbb{C}))$ and up to a scalar multiple,
we have that a~Rota---Baxter operator $R$ of weight 0 on $\Cur(\sl_2(\mathbb{C}))$
is either a~$\partial$-linear extension of an RB-operator of weight 0 on $\sl_2(\mathbb{C})$ or one of the following for some nonzero $q(\partial)\in\mathbb{C}[\partial]$:

(R1) $R(e) = 0$, $R(f) = q(\partial)e + \alpha h$, $R(h) = - 2\alpha e$, $\alpha\in \mathbb{C}$;

(R2) $R(e) = R(f) = 0$, $R(h) = q(\partial)h$.
\end{theorem}

\begin{proof}
Above, we have obtained the RB-operators $R_1$, $R_2$, $R_3$, and $R_4$.
The RB-operator $R_1$ coincides with (R1). 

For $R_2$, define an automorphism $\psi$ of $\sl_2(\mathbb{C})$ as follows,
$$
\psi(e) = e,\quad 
\psi(f) = f - (\alpha/2)h - (\alpha^2/4)e,\quad
\psi(h) = \alpha e+h.
$$
Then $\psi^{-1}R_2\psi$ is exactly~(R2). 

For $R_3$, $\xi^{-1}R_3\xi$ gives again~(R2). Here 
$$
\xi(e) = \frac{\sqrt{\alpha}h-e+\alpha f}{2\sqrt{\alpha}},\quad 
\xi(f) = \frac{\sqrt{\alpha}h+e-\alpha f}{2\sqrt{\alpha}},\quad 
\xi(h) = \frac{e+\alpha f}{\sqrt{\alpha}}.
$$

Finally, consider the RB-operator $R_4$.
If $\alpha+\beta^2=0$, then $\pi^{-1}R_4\pi$ equals to a particular case of~(R1),
where 
$$
\pi(e) = -\beta^2 e + f + \beta h,\quad \pi(f) = e,\quad \pi(h) = 2\beta e -h.
$$
Otherwise, $\theta^{-1}R_4\theta$ is exactly~(R2). Here
\begin{gather*}
\theta(e) 
 = \frac{\alpha e-f+(\beta+D)(2\beta e-h)}{2iD},\quad 
\theta(f) 
 = \frac{\alpha e-f+(\beta-D)(2\beta e-h)}{2iD},\\
\theta(h) = \frac{\alpha e + f + \beta h}{D}, \quad
D = \sqrt{\alpha+\beta^2}. 
 \qedhere
\end{gather*}
\end{proof}

\subsection{Connection with conformal CYBE}

In~\cite{Bai}, it was shown that every solution to the conformal classical Yang---Baxter equation on 
a conformal Lie algebra~$L$ endowed with a non-degenerate symmetric invariant conformal bilinear
form gives rise to a Rota---Baxter
operator on $L$.

\begin{proposition}[{\cite[Corollary 3.3]{Bai}}]
Let $L$ be a finite Lie conformal algebra which is free as a $C[\partial]$-module.
Suppose that there exists a non-degenerate symmetric invariant conformal bilinear
form on~$L$, and $r \in L \otimes L$ is skew-symmetric. Then $r$ is a solution to the
CCYBE if and only if $P^r_0$ is a Rota---Baxter operator of weight~0 on~$L$ for 
$P^r \in \mathrm{Cend}(R)$ defined by the formula
\begin{equation} \label{CCYBE2RB}
\langle r,u\otimes v\rangle_{(\lambda,\mu)}
 = \langle P^r_{\lambda-\partial}(u),v\rangle_\mu.
\end{equation}
\end{proposition}

Let us clarify which Rota---Baxter operators on $L = \Cur(\sl_2(\mathbb{C}))$ we get from the solutions to
the weak conformal classical Yang---Baxter equation on~$L$.

\begin{theorem}[\cite{Cursl2}] \label{Thm:Cursl2}
Let $L = \Cur(\sl_2(\mathbb{C}))$,
$r = \sum A_{ql}(\partial_{\otimes 1}, 
\partial_{\otimes 2}) q\otimes l\in L\otimes L$, 
$q, l \in \{e, f, h\}$, be a skew-symmetric solution to CCYBE. 
Then $A_{qq}(0,0) = 0$ for any $q$ and 
$$
A_{fe}(0,0) = - A_{ef}(0,0) = \beta, \quad
A_{he}(0,0) = -A_{eh}(0,0) = \alpha, \quad
A_{hf}(0,0) = -A_{fh}(0,0) = \gamma
$$
for some $\alpha,\beta,\gamma\in\mathbb{C}$.
Moreover, $A_{ql}(x,-x) - A_{ql}(0,0) = a_{ql} x f(x^2)$ for a unitary polynomial $f(x)$ and $a_{q,l} \in \mathbb{C}$, and 
up to action of automorphisms of~$L$, we have three cases:

(i) $a_{ee} = 1$, $a_{ql} = 0$ for $(q,l)\neq(e,e)$, and $\beta = \gamma = 0$;

(ii) $a_{hh} = \lambda\in\mathbb{C}\setminus\{0\}$, $a_{ql} = 0$ for $(q,l)\neq(h,h)$, and $\alpha =  \beta = \gamma = 0$; 

(iii) $a_{ql} = 0$ for all $q,l$.
\end{theorem}

Recall that given a~semisimple finite-dimensional Lie algebra $\mathfrak{g}$ over $\mathbb{C}$,
we define the required form on $\Cur(\mathfrak{g})$ as follows.
Firstly, $\langle a,b\rangle_\lambda:=\langle a,b\rangle$, i.\,e. the Killing form on $\mathfrak{g}$. Secondly, extend it by the rule
$$
\langle \partial a,b\rangle 
 = -\lambda \langle a,b\rangle
 = - \langle a, \partial b\rangle,\quad a,b\in L. 
$$
Such form is invariant, it means that 
$\langle [a_\mu b],c\rangle_\lambda 
 = \langle a,[b_{\lambda-\partial}c]\rangle_\mu$ for all $a,b,c\in \mathfrak{g}$.
We additionally set for $a\otimes b,c\otimes d\in L\otimes L$,
$$
\langle a\otimes b,c\otimes d\rangle_{(\lambda,\mu)}
 = \langle a,c\rangle_\lambda 
   \langle b,d\rangle_\mu. 
$$

Let 
$r = \sum\limits_{i,j=1}^3\sum\limits_{k,l\geq0}\gamma_{kl}^{ij} \partial^k i\otimes \partial^l j$ be a skew-symmetric solution to CCYBE.
Take $u,v\in\{e,f,h\}$, then 
$$
\langle r,u\otimes v\rangle_{(\lambda,\mu)}
 = \sum\limits_{i,j=1}^3\sum\limits_{k,l\geq0}\gamma_{kl}^{ij} 
 \langle \partial^k i, u\rangle_\lambda
 \langle \partial^l j, v\rangle_\mu 
 = \sum\limits_{i,j=1}^3\sum\limits_{k,l\geq0}\gamma_{kl}^{ij}(-\lambda)^k
 \langle i, u\rangle_\lambda \langle \partial^l j, v\rangle_\mu.
$$
Since $v$ is arbitrary and the Killing form is non-degenerate on~$\sl_2(\mathbb{C})$, we have by~\eqref{CCYBE2RB}
$$
P^r_{\lambda-\partial}(u)
 = \sum\limits_{i,j=1}^3\sum\limits_{k,l\geq0}\gamma_{kl}^{ij}(-\lambda)^k
 \langle i, u\rangle_\lambda \partial^l j.
$$
For $\lambda = \partial$, we get in terms of the polynomials $A_{ij}(\partial_{\otimes1},\partial_{\otimes2})$,
$$
P^r_0(u)
 = \sum\limits_{i,j=1}^3\sum\limits_{k,l\geq0}\gamma_{kl}^{ij}(-\partial)^k\partial^l \langle i, u\rangle_\lambda j
 = \sum\limits_{i,j=1}^3 A_{ij}(-\partial,\partial)\langle i, u\rangle j.
$$ 
By Theorem~\ref{Thm:Cursl2}, we get either an RB-operator which is a $\partial$-linear extension of an RB-ope\-rator of weight~0 on $\sl_2(\mathbb{C})$ or
both RB-operators (R1) and~(R2) with odd polynomial~$q(\partial)$.

\section{Rota---Baxter operators of weight~1 on $\Cur(\sl_2(\mathbb{C}))$}

Let $R$ be an RB-operator of weight~1 on $\Cur(\sl_2(\mathbb{C}))$.
We again use the formulas~\eqref{ROnBasis}
and repeat the computations similar to the ones from the case of weight~0.
Thus, we obtain the~system:
\begin{multline}\label{1'}
c_e(\lambda+\partial)c_h(-\lambda) - c_e(-\lambda)c_h(\lambda+\partial) \\
 = a_e(\partial)(c_e(\lambda+\partial)-c_e(-\lambda)) + b_e(\partial)(c_f(-\lambda)-c_f(\lambda+\partial)),
\end{multline}

\vspace{-0.9cm}
\begin{multline}\label{2'}
c_f(-\lambda)c_h(\lambda+\partial) - c_f(\lambda+\partial)c_h(-\lambda) \\
 = a_f(\partial)(c_e(\lambda+\partial)-c_e(-\lambda)) + b_f(\partial)(c_f(-\lambda)-c_f(\lambda+\partial)),
\end{multline}

\vspace{-0.9cm}
\begin{multline}\label{3'}
c_e(-\lambda)c_f(\lambda+\partial) - c_e(\lambda+\partial)c_f(-\lambda) \\
 = 2a_h(\partial)(c_e(\lambda+\partial)-c_e(-\lambda)) + 2b_h(\partial)(c_f(-\lambda)-c_f(\lambda+\partial)),
\end{multline}

\vspace{-0.9cm}
\begin{multline}\label{4'}
2a_e(\lambda+\partial)a_h(-\lambda) - 2a_e(-\lambda)a_h(\lambda+\partial) \\
 = c_e(\partial)(a_f(\lambda+\partial)-a_f(-\lambda)) + 2a_e(\partial)(a_h(-\lambda)-a_h(\lambda+\partial)),
\end{multline}

\vspace{-0.9cm}
\begin{multline}\label{5'}
2a_f(-\lambda)a_h(\lambda+\partial) - 2a_f(\lambda+\partial)a_h(-\lambda) \\
 = c_f(\partial)(a_f(\lambda+\partial)-a_f(-\lambda)) + 2a_f(\partial)(a_h(-\lambda)-a_h(\lambda+\partial)),
\end{multline}

\vspace{-0.9cm}
\begin{multline}\label{6'}
a_e(-\lambda)a_f(\lambda+\partial) - a_e(\lambda+\partial)a_f(-\lambda) \\
 = c_h(\partial)(a_f(\lambda+\partial)-a_f(-\lambda)) + 2a_h(\partial)(a_h(-\lambda)-a_h(\lambda+\partial)),
\end{multline}

\vspace{-0.9cm}
\begin{multline}\label{7'}
2b_e(\lambda+\partial)b_h(-\lambda) - 2b_e(-\lambda)b_h(\lambda+\partial) \\
 = c_e(\partial)(b_e(-\lambda)-b_e(\lambda+\partial)) + 2b_e(\partial)(b_h(\lambda+\partial)-b_h(-\lambda)),
\end{multline}

\vspace{-0.9cm}
\begin{multline}\label{8'}
2b_f(-\lambda)b_h(\lambda+\partial) - 2b_f(\lambda+\partial)b_h(-\lambda) \\
 = c_f(\partial)(b_e(-\lambda)-b_e(\lambda+\partial)) + 2b_f(\partial)(b_h(\lambda+\partial)-b_h(-\lambda)),
\end{multline}

\vspace{-0.9cm}
\begin{multline}\label{9'}
b_e(-\lambda)b_f(\lambda+\partial) - b_e(\lambda+\partial)b_f(-\lambda) \\
 = c_h(\partial)(b_e(-\lambda)-b_e(\lambda+\partial)) + 2b_h(\partial)(b_h(\lambda+\partial)-b_h(-\lambda)),
\end{multline}

\vspace{-0.9cm}
\begin{multline}\label{10'}
2(a_h(-\lambda)b_e(\lambda+\partial)-a_e(-\lambda)b_h(\lambda+\partial)) \\
 = -2b_h(\lambda+\partial)a_e(\partial) - 2a_h(-\lambda)b_e(\partial) + (a_e(-\lambda)+b_f(\lambda+\partial) +1)c_e(\partial),
\end{multline}

\vspace{-0.9cm}
\begin{multline}\label{11'}
- 2(a_h(-\lambda)b_f(\lambda+\partial)-a_f(-\lambda)b_h(\lambda+\partial)) \\
 = -2b_h(\lambda+\partial)a_f(\partial) - 2a_h(-\lambda)b_f(\partial) + (a_e(-\lambda)+b_f(\lambda+\partial) +1)c_f(\partial),
\end{multline}

\vspace{-0.9cm}
\begin{multline}\label{12'}
a_e(-\lambda)b_f(\lambda+\partial)-a_f(-\lambda)b_e(\lambda+\partial) \\
 = -2b_h(\lambda+\partial)a_h(\partial) - 2a_h(-\lambda)b_h(\partial) + (a_e(-\lambda)+b_f(\lambda+\partial) +1)c_h(\partial),
\end{multline}

\vspace{-0.9cm}
\begin{multline}\label{13'}
2(a_h(-\lambda)c_e(\lambda+\partial)-a_e(-\lambda)c_h(\lambda+\partial)) \\
 = -2( a_e(-\lambda) + c_h(\lambda+\partial) + 1)a_e(\partial) + 2a_f(-\lambda)b_e(\partial) + c_f(\lambda+\partial)c_e(\partial),
\end{multline}

\vspace{-0.9cm}
\begin{multline}\label{14'}
- 2(a_h(-\lambda)c_f(\lambda+\partial)-a_f(-\lambda)c_h(\lambda+\partial)) \\
 = -2( a_e(-\lambda) + c_h(\lambda+\partial) + 1)a_f(\partial) + 2a_f(-\lambda)b_f(\partial) + c_f(\lambda+\partial)c_f(\partial),
\end{multline}

\vspace{-0.9cm}
\begin{multline}\label{15'}
a_e(-\lambda)c_f(\lambda+\partial)-a_f(-\lambda)c_e(\lambda+\partial) \\
 = -2( a_e(-\lambda) + c_h(\lambda+\partial) + 1)a_h(\partial) + 2a_f(-\lambda)b_h(\partial) + c_f(\lambda+\partial)c_h(\partial),
\end{multline}

\vspace{-0.9cm}
\begin{multline}\label{16'}
2(b_h(-\lambda)c_e(\lambda+\partial)-b_e(-\lambda)c_h(\lambda+\partial)) \\
 = -2b_e(-\lambda)a_e(\partial) + 2(b_f(-\lambda)+c_h(\lambda+\partial) + 1)b_e(\partial) - c_e(\lambda+\partial)c_e(\partial),
\end{multline}

\vspace{-0.9cm}
\begin{multline}\label{17'}
- 2(b_h(-\lambda)c_f(\lambda+\partial)-b_f(-\lambda)c_h(\lambda+\partial)) \\
 = -2b_e(-\lambda)a_f(\partial) + 2(b_f(-\lambda)+c_h(\lambda+\partial) + 1)b_f(\partial) - c_e(\lambda+\partial)c_f(\partial),
\end{multline}

\vspace{-0.9cm}
\begin{multline}\label{18'}
b_e(-\lambda)c_f(\lambda+\partial)-b_f(-\lambda)c_e(\lambda+\partial) \\
 = -2b_e(-\lambda)a_h(\partial) + 2(b_f(-\lambda)+c_h(\lambda+\partial) + 1)b_h(\partial) - c_e(\lambda+\partial)c_h(\partial).
\end{multline}

We can extend an automorphism of $\sl_2(\mathbb{C})$
to an automorphism of $L = \Cur(\sl_2(\mathbb{C}))$.
By Proposition~\ref{prop:RBOnsl2-1} and by Corollary~\ref{coro} we may assume that
\begin{equation}\label{R0-sl2-nonzero}
a_f(0) = 0, \quad
b_e(0) = b_f(0) = b_h(0) = 0, \quad c_e(0) = c_f(0) = 0.
\end{equation}

\begin{theorem} \label{Thm:NonzeroCase}
Up to conjugation with an automorphism of $\Cur(\sl_2(\mathbb{C}))$ we have that a~nontrivial Rota---Baxter operator $R$ of weight~1 on $\Cur(\sl_2(\mathbb{C}))$ is either 
a~$\partial$-linear extension of an RB-operator of weight~1 on $\sl_2(\mathbb{C})$ or the following one for some nonzero $q(\partial)\in\mathbb{C}[\partial]$:

(Q1) $R(e) = -e$, $R(f) = 0$, $R(h) = q(\partial)h$.
\end{theorem}

\begin{proof}
By~\eqref{6'},~\eqref{7'},~\eqref{9'},~\eqref{11'},~\eqref{12'}, and~\eqref{17'} considered at $\lambda+\partial=0$, we have
\begin{gather}
a_f(\partial)(c_h(\partial)-a_e(0) )
 = 2a_h(\partial)(a_h(\partial)-a_h(0)), \label{afch=2ah^2} \\
b_e(\partial)(c_e(\partial)-2b_h(\partial)) = 0, \label{be(ce-2bh)} \\
b_e(\partial)c_h(\partial) = 2b_h^2(\partial), \label{bech-2bh^2} \\
(a_e(\partial)+1)c_f(\partial)
 = 2a_h(\partial)b_f(\partial), \label{(ae+1)cf=2ahbf} \\
(a_e(\partial)+1)c_h(\partial)
 = 2a_h(\partial)b_h(\partial), \label{(ae+1)ch=2ahbh} \\
a_f(\partial)b_e(\partial)
 = b_f(\partial)(b_f(\partial)+1). \label{afbe-bf2}
\end{gather}
Also, by~\eqref{14'}--\eqref{16'} with $\lambda = 0$, we get
\begin{gather}
2(a_e(0)+1+c_h(\partial))a_f(\partial)
 = c_f(\partial)(c_f(\partial)+2a_h(0)), \label{14''} \\
2(a_e(0)+1+c_h(\partial))a_h(\partial)
 = c_f(\partial)(c_h(\partial)-a_e(0)), \label{15''} \\
2b_e(\partial)(c_h(\partial)+1) = c_e^2(\partial) \label{2be(ch+1)=ce2}.
\end{gather}

{\sc Case I}: $b_e(\partial) = 0$.
By~\eqref{bech-2bh^2},~\eqref{afbe-bf2}, and~\eqref{2be(ch+1)=ce2}, we have $b_h(\partial) = b_f(\partial) = c_e(\partial) = 0$. By~\eqref{13'}, we conclude that
$a_e(\partial)(a_e(\partial)+1) = 0$.

{\sc Case IA}: $a_e(\partial) = 0$.
Then the formulas~\eqref{(ae+1)cf=2ahbf} and~\eqref{(ae+1)ch=2ahbh} imply $c_f(\partial) = c_h(\partial) = 0$. Further, by~\eqref{14''} and~\eqref{15''}, $a_f(\partial) = a_h(\partial) = 0$, i.\,e., $R = 0$.

{\sc Case IB}: $a_e(\partial) = -1$.

{\sc Case IBA}: $c_f(\partial) = 0$.
The equalities~\eqref{14''} and~\eqref{15''} give $a_f(\partial)c_h(\partial) = a_h(\partial)c_h(\partial) = 0$. If $c_h(\partial)\neq0$, then
$a_f(\partial) = a_h(\partial) = 0$ and it is~(Q1).

Suppose that $c_h(\partial) = 0$.
Consider~\eqref{5'} and~\eqref{6'} with $\lambda+\partial=0$:
$$
a_f(\partial)(a_h(\partial)-2a_h(0)) = 0, \quad
a_f(\partial) = 2a_h(\partial)(a_h(\partial)-a_h(0)).
$$
They imply that $a_h(\partial)$ and $a_f(\partial)$ are constant.
Thus, $R$ is defined by Proposition~\ref{prop:RBOnsl2-1}.

{\sc Case IBB}: $c_f(\partial) \neq 0$.
From the equality~\eqref{2'}, it follows that
$c_h(\partial) = \alpha c_f(\partial)$ for some $\alpha\in\mathbb{C}$.
By~\eqref{15''}, we write down
$$
2a_h(\partial)c_h(\partial)
 = 2\alpha a_h(\partial)c_f(\partial)
 = c_f(\partial)(1+\alpha c_f(\partial)).
$$
If $\alpha = 0$, then $c_f(\partial) = 0$, a contradiction.
Hence, $a_h(\partial) = \frac{\alpha c_f(\partial)+1}{2\alpha}$
and $a_h(0) = 1/(2\alpha)$.
By~\eqref{14''}, we derive that
$$
a_f(\partial)
 = \frac{1}{2\alpha}\left(c_f(\partial) + \frac{2}{2\alpha}\right)
 = \frac{\alpha c_f(\partial)+1}{2\alpha^2}.
$$
Therefore, $0 = a_f(0) = 1/(2\alpha^2)$, a contradiction.

{\sc Case II}: $b_e(\partial)\neq 0$.
Then by~\eqref{be(ce-2bh)}, we deduce that
$c_e(\partial) = 2b_h(\partial)$.
Due to~\eqref{bech-2bh^2}, we get
$c_h(\partial) = c_e^2(\partial)/(2b_e(\partial))$.
Finally, we apply this equality to~\eqref{2be(ch+1)=ce2},
$$
c_e^2(\partial)
 = 2b_e(\partial)(c_h(\partial)+1)
 = 2b_e(\partial)\left(\frac{c_e^2(\partial)}{2b_e(\partial)}+1\right)
 = c_e^2(\partial) + 2b_e(\partial),
$$
a contradiction.
\end{proof}

We may generalize the RB-operator (Q1) from Theorem~\ref{Thm:NonzeroCase} for 
$L = \Cur(\mathfrak{g})$, where $\mathfrak{g}$ is a finite-dimensional semisimple Lie algebra.

\begin{example}
Let $\mathfrak{g}$ be a finite-dimensional semisimple Lie algebra over $\mathbb{C}$
with a~root system~$\Phi$.
Let a linear operator $R$ acts on $L = \Cur(\mathfrak{g})$ as follows,
$R(\mathfrak{h})\subset \mathbb{C}[\partial]\mathfrak{h}$ for the Cartan subalgebra $\mathfrak{h}$ of $\mathfrak{g}$, $R(e_\lambda) = -e_\lambda$, when $\lambda\in\Phi_+$, and $R(e_\lambda) = 0$ for all $\lambda\in\Phi_-$.
Then $R$ is a Rota---Baxter operator of weight~1 on $L$.
\end{example}

\section*{Acknowledgements}

The authors are supported by the grant of the President of the Russian
Federation for young scientists (MK-1241.2021.1.1).

\medskip
\noindent Vsevolod Gubarev \\
Roman Kozlov \\
Sobolev Institute of Mathematics \\
Acad. Koptyug ave. 4, 630090 Novosibirsk, Russia \\
Novosibirsk State University \\
Pirogova str. 2, 630090 Novosibirsk, Russia \\
e-mail: wsewolod89@gmail.com, dyadkaromka94@gmail.com

\end{document}